\documentclass[a4paper,11pt,reqno]{amsart}

\usepackage{amsmath, amsthm, amssymb}
\usepackage{mathtools}      
\usepackage{mathabx}        
\usepackage[bb=fourier,cal=euler,scr=rsfs]{mathalfa}	
\usepackage{enumitem}       

\usepackage[utf8]{inputenc}

\usepackage{tikz-cd}

\usepackage{comment}

\newtheorem*{thm*}{Theorem}

\newtheorem{thm}{Theorem}[section]

\newtheorem{cor}[thm]{Corollary}

\newtheorem{lemma}[thm]{Lemma}
\newtheorem{prop}[thm]{Proposition}

\newtheorem{teoA}{Theorem}

\newtheorem{propB}[teoA]{Theorem}

\newtheorem{teoAprima}{Theorem}

\theoremstyle{definition}
\newtheorem{definition}[thm]{Definition}
\theoremstyle{remark}
\newtheorem{remark}[thm]{Remark}
\newtheorem*{notation}{Notation}

\newtheorem{example}[thm]{Example}

\newcommand{\W}{\mathcal{W}}
\newcommand{\F}{\mathcal{F}}

\DeclareMathOperator{\id}{id}

\DeclareMathOperator{\length}{length}

\usepackage{xcolor}
\usepackage{hyperref}
\hypersetup{
    colorlinks,
    linkcolor={red!50!black},
    citecolor={blue!50!black},
    urlcolor={blue!80!black}
}

\begin{document}

\author[Nancy Guelman]{Nancy Guelman} 
\address{IMERL, Facultad de Ingeniería, Universidad de la Rep\'ublica, Uruguay}
\email{nguelman@fing.edu.uy}

\author[Santiago Martinchich]{Santiago Martinchich} 
\address{CMAT, Facultad de Ciencias, Universidad de la Rep\'ublica, Uruguay \vspace{-0.25cm}}

\address{LMO, Université Paris-Saclay, 91405, Orsay, France}
\email{smartinchich@cmat.edu.uy}

\title[Uniqueness of minimal unstable lamination]{Uniqueness of minimal unstable lamination for discretized Anosov flows}


\begin{abstract}
We consider the class of partially hyperbolic diffeomorphisms $f:M\to M$ obtained as the discretization of topological Anosov flows. We show uniqueness of minimal unstable lamination for these systems provided that the underlying Anosov flow is transitive and not orbit equivalent to a suspension. As a consequence, uniqueness of quasi-attractors is obtained.
If the underlying Anosov flow is not transitive we get an analogous finiteness result provided that the restriction of the flow to any of its attracting basic pieces is not a suspension.
A similar uniqueness result is also obtained for certain one-dimensional center skew-products.
\end{abstract}

\maketitle

\section{Introduction} 

A diffeomorphism $f$ in a closed manifold $M$ is called \emph{partially hyperbolic} if there exists a $Df$-invariant continuous decomposition $$TM=E^s \oplus E^c \oplus E^u$$ such that vectors in $E^s$ and $E^u$ are uniformly contracted for future and past iterates of $f$, respectively, and vectors in $E^c$ have an intermediate behavior (a precise definition is given in Section \ref{prelim1}). 

A partially hyperbolic diffeomorphism $f:M\to M$ with $\dim(E^c)=1$ is a \emph{discretized Anosov flow} if there exist a \emph{topological Anosov flow} $\varphi_t:M\to M$ (definition given in Section \ref{prelim2}) and a continuous function $\tau:M\to \mathbb{R}_{>0}$ such that $$f(x)=\varphi_{\tau(x)}(x)$$ for every $x$ in $M$.

Discretized Anosov flows form a class of partially hyperbolic diffeomorphims that naturally contains $C^1$-perturbations of the time 1 map of Anosov flows (see Section \ref{prelim2}). But not only that, it has been recently shown in \cite{BFFP} that discretized Anosov flows form a somewhat large class, comprising the whole homotopic to the identity class of partially hyperbolic diffeomorphism of certain 3-manifolds.

This work fits in the study of the dynamics of discretized Anosov flows. In particular, on the finiteness and uniqueness of \emph{quasi-attractors} and \emph{quasi-repellers} (definition in Section \ref{prelim2}).  

It is well known that the bundles $E^u$ and $E^s$ uniquely integrate to $f$-invariant foliations (see e.g. \cite{HPS}). We denote them as the \emph{unstable} $\W^u$ and \emph{stable} $\W^s$ foliations, respectively. Since quasi-attractors are pairwise disjoint compact sets saturated by $\W^u$-leaves then each one of them contains at least one minimal set for the foliation $\W^u$. Thus, uniqueness (resp. finiteness) of \emph{minimal unstable laminations} implies uniqueness (resp. finiteness) of quasi-attractors.

Finiteness of minimal unstable laminations is obtained in \cite{CPS} for a $C^1$-open and dense subset of partially hyperbolic diffeomorphisms with one-dimensional center. In this text we aim to a more global (non-perturbative) study involving uniqueness/finiteness results for whole classes of examples.

It is worth pointing out that we focus on minimal unstable laminations and quasi-attractors but the results have obvious analogous statements for minimal stable laminations and quasi-repellers.

\subsection{Discretized transitive Anosov flows}

Discretized Anosov flows with arbitrary number of attractors and repellers can be obtained by perturbing the time 1 map of an Anosov diffeomorphism's suspension: as $M$ fibers over the circle and $\varphi_1$ preserves fibers one can perturb $\varphi_1$ so that it becomes Morse-Smale or even a dynamics with infinitely many quasi-attractors in the base (see Example \ref{exSusp} for details).

In \cite{BG2} examples of Axiom A discretized Anosov flows $f(x)=\varphi_{\tau(x)}(x)$ with a proper attractor and a proper repeller have been built as a discretization of any transitive Anosov flow $\varphi_t:M\to M$ provided that $\dim(M)=3$. 

Recall that two flows are said to be \emph{orbit equivalent} whenever there exists a homeomorphism taking orbits of one into orbits of the other and preserving its orientation. We obtain:

\begin{teoA}\label{thmA}
Let $f(x)=\varphi_{\tau(x)}(x)$ be a discretized Anosov flow. Suppose $\varphi_t$ is transitive and not orbit equivalent to a suspension. Then $f$ has a unique minimal unstable lamination.
\end{teoA}

\begin{cor}\label{corA}
Any $f$ as in Theorem \ref{thmA} has at most one quasi-attractor.
\end{cor}

Theorem \ref{thmA} is already known from \cite{HU} for discretized Anosov flows in a $C^1$-neighborhood of the time 1 map of a transitive Anosov flow that is not orbit equivalent to a suspension. 

We point out that our proof relies on a different approach. The main inspiration for the present work comes from \cite{BG1} where it was shown that every Axiom A discretized Anosov flow as in the hypothesis of Theorem \ref{thmA} admits a unique attractor. By generalizing the arguments in \cite{BG1} (see also \cite{Gue}) we are able to remove the Axiom A hypothesis and to obtain not only uniqueness of quasi-attractor but also of minimal unstable lamination. 

In the case $f$ is chain-transitive Corollary \ref{corA} gives us no new information but uniqueness of minimal unstable lamination is still interesting. It implies, for example, that the supports of all $u$-Gibbs measures have non-trivial intersection as the support of any such a measure is a $\W^u$-saturated compact set. In \cite[Theorem 1.2]{HU} more precise consequences are obtained.

\subsection{Discretized non-transitive Anosov flows}
For $f(x)=\varphi_{\tau(x)}(x)$ a discretized Anosov flow such that $\varphi_t$ is not transitive the problem reduces to study the behavior of $\varphi_t$ on the attracting basic pieces of $\varphi_t$.

For example, the time 1 map of the Franks-Williams's  non-transitive Anosov flow \cite{FW} can be perturbed to obtain arbitrary number of quasi-attractors (see Example \ref{exFW}). The unique attractor $\Lambda$ in this example verifies that $\varphi_t|_\Lambda:\Lambda\to \Lambda$ is orbit equivalent to a suspension so one can essentially perform in a neighborhood of $\Lambda$ the same perturbation as mentioned above for the time 1 map of an Anosov's suspension.

On the other hand, the arguments for obtaining Theorem \ref{thmA} are also valid in restriction to any non-suspension basic attracting piece. We obtain:

\begin{teoAprima}\label{thmA'}
Let $f(x)=\varphi_{\tau(x)}(x)$ be a discretized Anosov flow. Then every attracting basic piece $\Lambda$ of $\varphi_t$ such that $\varphi_t|_\Lambda$ is not orbit equivalent to a suspension admits a unique minimal unstable lamination for $f$.
\end{teoAprima}

\begin{cor}\label{corA'}
Let $f(x)=\varphi_{\tau(x)}(x)$ be a discretized Anosov flow. Suppose that all attracting basic pieces $\Lambda_1$, \ldots, $\Lambda_k$ of $\varphi_t$ verify that $\varphi_t|_{\Lambda_i}$ is not orbit equivalent to a suspension. Then $f$ has exactly $k$ minimal unstable laminations (and exactly $k$ quasi-attractors). Moreover, each one of them is contained in one of the attracting basic pieces $\Lambda_1$, \ldots, $\Lambda_k$.
\end{cor}

Discretized non-transitive Anosov flows in the hypothesis of Corollary \ref{corA'} can be constructed using the techniques from \cite{FW} (see also \cite{BBY}). We briefly sketch their construction in Section \ref{section6}.

\subsection{Skew-products}

We say that $f:M\to M$ is a \emph{ partially hyperbolic skew-product} if it admits an $f$-invariant center foliation $\W^c$ such that $M$ is a fiber bundle with $M/\W^c$ as base and the leaves of $\W^c$ as fibers. If $\dim(E^c)=1$, we say that $(M,\W^c)$ is the \emph{trivial bundle} if $\W^c$ is topologically equivalent to the foliation $\{\cdot\}\times S^1$ in $M/\W^c\times S^1$. We say that $(M,\W^c)$ is a \emph{virtually trivial bundle} if it is the trivial bundle modulo finite cover. Examples of non-trivial skew-products that are virtually trivial and not virtually trivial can be found in \cite{BW}.

The proof of Theorem \ref{thmA} will follow from the more general statements of Proposition \ref{prop1} and Proposition \ref{prop2} (see Section \ref{sectionproofs}). As a consequence of these propositions we recover also the uniqueness of minimal unstable lamination result of \cite{HP} when the bundle is non-trivial in dimension 3 and we extend it to any dimension:

\begin{propB}\label{propB}
Let $f:M\to M$ be a partially hyperbolic skew-product with $\dim(E^c)=1$ such that the induced dynamics in the space of center leaves, $F:M/\W^c\to M/\W^c$, is transitive. If $(M,\W^c)$ is not a virtually trivial bundle then 
$f$ admits a unique minimal unstable lamination and a unique quasi-attractor.
\end{propB}

In fact, Theorem \ref{propB} is also valid if we exchange the hypothesis `skew-product' for `$\W^c$ uniformly compact'. We will precise this in Section \ref{sectionproofs}.

In dimension 3, examples of partially hyperbolic skew-products with a proper attractor and a proper repeller such that $\W^c$ is given by the fibers of any non-trivial bundle over $M/\W^c=\mathbb{T}^2$ are constructed in \cite{Shi}.

It is worth noting the marked correspondence between skew-products and discretized Anosov flows concerning the uniqueness and existence results. The trivial bundle case, the uniqueness of minimal unstable lamination result of \cite{HP} (extended in Theorem \ref{propB}) and the examples of \cite{Shi} mirror the suspension case, Theorem \ref{thmA} and the examples of \cite{BG2}, respectively.

Notice that the hypothesis `$F:M/\W^c\to M/\W^c$ transitive' in Theorem \ref{propB} is somewhat natural since in this setting $F$ is a \emph{topological Anosov homeomorphism} that preserves two topologically transverse contracting/expanding continuous foliations $\W^{cs}|_{M/\W^c}$ and $\W^{cu}|_{M/\W^c}$. A potential Theorem B', in analogy with Theorem \ref{thmA'}, would involve dealing with attracting basic pieces of a non-transitive $F$. 

\subsection{Classical examples in dimension 3 and beyond}

The `classical examples' (in the sense of Pujals's conjecture and \cite{BW}) of partially hyperbolic diffeomophisms in dimension 3 are skew-products,  \emph{deformations of Anosov diffeomorphisms} (those that are homotopic to Anosov in $\mathbb{T}^3$) and discretized Anosov flows.

For deformations of Anosov diffeomorphisms uniqueness of minimal stable and unstable lamination is proved in \cite{Po}. Existence of a proper quasi-attractor is unknown (see \cite[Question 2]{Po}).

With Theorem \ref{thmA} and Theorem \ref{thmA'} we complete in a certain sense the uniqueness and finiteness problem for the classical examples in dimension 3 modulo the structure of $\W^c$. In particular, the existence of infinetely many minimal unstable laminations is always associated with a region (the whole manifold or some proper attracting region) where $\W^c$ `looks like' a suspension flow.

Beyond the classical examples, the first non-dynamically coherent examples were obtained in \cite{HHU}. They identified the existence of a periodic torus tangent to $E^s\oplus E^c$ or $E^c\oplus E^u$ as a possible obstruction for dynamically coherence. Notice that such a torus is necessary an attractor or a repeller. In \cite{HP2} it was shown that examples with this type of tori have periodic regions homeomorphic to $\mathbb{T}^2\times (0,1)$ in the complement of these tori. In these regions $E^c$ integrates to $f$-invariant `interval fibers' transverse to $\mathbb{T}^2\times\{\cdot\}$ and the dynamics is of the type Anosov times identity. So essentially the existence of minimal unstable laminations or quasi-attractors for this type of examples is similar to the trivial skew-product and suspension's of Anosov map scenarios.

More recently, the realm of classical examples has been enlarged with new challenging examples (see for example \cite{BPP}, \cite{BGP} and \cite{BGHP}). It is natural to ask if results of uniqueness/finiteness of minimal unstable laminations/quasi-attractors are also valid for whole classes of these new examples.

\subsection{Acknowledgments} The authors would like to thank S. Crovisier and R. Potrie (co-directors of the second author) for their constant attention on the progress of this work and fruitful discussions.

S. Martinchich was partially supported by `ADI 2019' project funded by IDEX Paris-Saclay ANR-11-IDEX-0003-02, Proyecto CSIC `Iniciación a la Investigación' No 133 and Proyecto CSIC No 618 `Sistemas Dinámicos'.

\section{Preliminaries}\label{preliminaries}

\subsection{Partial hyperbolicity and quasi-attractors}\label{prelim1}

We say that a $C^1$-diffeomorphism $f$ in a closed Riemannian manifold $M$ is \emph{(strongly) partially hyperbolic} if it preserves a continuous splitting $TM=E^s\oplus E^c \oplus E^u$, with non-trivial \emph{stable} $E^s$ and \emph{unstable} $E^u$ bundles, such that for some constants $\lambda\in (0,1)$ and $C>0$:
\begin{center}$||Df^n_{E^s(x)}||< C \lambda^n$, \hspace{0.1cm} $||Df^{-n}_{E^u(x)}||< C \lambda^n$ \hspace{0.1cm} and

$||Df_{E^s(x)}||<m(Df_{E^c(x)})\leq ||Df_{E^c(x)}||< m(Df_{E^u(x)})$
\end{center}
for every $x\in M$ and $n\geq 0$. Recall that for $T$ linear, $m(T)$ and $||T||$ denote the min and max of $\{||T(v)||\}_{||v||=1}$, respectively. For simplicity, all along this text we will drop the word `strongly' and call these systems merely \emph{partially hyperbolic}.

The strong bundles $E^s$ and $E^u$ are known to uniquely integrate to $f$-invariant foliations $\W^s$ and $\W^u$ (see \cite{HPS}). The bundles $E^s\oplus E^c$ and $E^c\oplus E^u$ may or may not be integrable. Whenever they integrate to $f$-invariant foliations (denoted by $\W^{cs}$ and $\W^{cu}$, respectively) we will say that $f$ is \emph{dynamically coherent}. If this is the case, then $\W^{cs}\cap \W^{cu}$ is an $f$-invariant foliation tangent to $E^c$ that we will denote by $\W^c$.

For every $i\in \{u,c,s,cs,cu\}$ the foliation $\W^i$ is a foliation with $C^1$ leaves tangent to a continuous bundle. We will consider the leaves of $\W^i$ with the intrinsic metric induced by the metric from $M$. For $R>0$ and $x\in M$ we denote by $\W^i_R(x)$ the ball of center $x$ and radius $R$ inside the leaf $\W^i(x)$. It is direct to check that if $x_n \xrightarrow{n} x$ then $\W^i_R(x_n) \xrightarrow{n} \W^i_R(x_n)$ in the Hausdorff topology. We will use this fact several times along the text.

We say that $A\subset M$ is a \emph{minimal unstable lamination} if it is a minimal set of the foliation $\W^u$, that is, if it is a $\W^u$-saturated compact set such that $\overline{\W^u(x)}=A$ for every $x\in A$. Minimal unstable laminations are minimal, with respect to the inclusion, among non-empty compact $\W^u$-saturated sets.

Let $\mathcal{R}(f)\subset M$ the \emph{chain recurrent set} of $f$, that is, the union of all points $x\in M$ such that there exists a non-trivial $\epsilon$-pseudo orbit from $x$ to $x$ for every $\epsilon>0$. It coincides with the complement of all points contained in a wandering region of the form $U\setminus f(\overline{U})$ for some open set $U$ such that $\overline{f(U)}\subset U$. We consider $\mathcal{R}(f)$ divided in equivalent classes, called \emph{chain recurrent classes}, given by the relation $x\sim y $ if and only if there exists a non-trivial $\epsilon$-pseudo orbit from $x$ to $y$ and another from $y$ to $x$ for every $\epsilon>0$. 

A \emph{quasi-attractor} is a chain recurrent class $A$ for which there exists a base of neighborhoods $\{U_i\}_i$ (i.e. $A\subset U_i$ and $A=\bigcap_i U_i$) such that $\overline{f(U_i)}\subset U_i$ for every $i$. Quasi-attractors always exists for homeomorphisms in compact metric spaces. A good reference for the notions of chain recurrent classes and quasi-attractors is \cite{CP}. 

Every quasi-attractor is a compact $\W^u$-saturated set so it contains at least one minimal unstable lamination. Then uniqueness (resp. finiteness) of minimal unstable laminations implies uniqueness (resp. finiteness) of quasi-attractors. 

\subsection{Discretized Anosov flows}\label{prelim2}
A general reference for this subsection is \cite[Appendix G]{BFFP}.

We say that a flow $\varphi_t:M\to M$ is an \emph{Anosov flow} if it preserves a continuous and $D\varphi_t$-invariant decomposition $TM=E^s\oplus E^c \oplus E^u$ such that vectors in $E^s$ and $E^u$ are uniformly contracted for the future and past, respectively, and the subbundle $E^c$ is generated by $\frac{\partial \varphi_t}{\partial t}|_{t=0}$.

We say that $\varphi_t:M\to M$ is a \emph{topological Anosov flow} if it is a continuous flow with  $\frac{\partial \varphi_t}{\partial t}|_{t=0}$ a continuous vector field and preserving two topologically transverse continuous foliations $\F^{ws}$ and $\F^{wu}$ such that:

\begin{itemize}
\item The foliaton $\F^{ws}\cap\F^{wu}$ is the foliation given by the orbits of $\varphi_t$.
\item Given $x$ in $M$ and $y\in \F^{ws}(x)$ (resp. $y\in \F^{wu}(x)$) there exists an increasing continuous reparametrization $h:\mathbb{R}\to \mathbb{R}$ such that $d(\varphi_t(x),\varphi_{h(t)}(y))\to 0$ as $t\to +\infty$ (resp. $t\to -\infty$).
\item There exists $\epsilon>0$ such that for every $x\in M$ and $y\in \F^{ws}_{loc}(x)$ (resp, $y\in \F^{wu}_{loc}(x)$), with $y$ not in the same orbit as $x$, and for every increasing continuous reparametrization $h:\mathbb{R}\to \mathbb{R}$ with $h(0)=0$, there exists $t\leq 0$ (resp. $t\geq 0$) such that $d(\varphi_t(x),\varphi_{h(t)}(y))>\epsilon$.
\end{itemize}

\begin{remark}\label{rmkAF} Most of the classical properties of Anosov flows are valid also in the context of topological Anosov flows (see \cite{Bar} and references therein). In particular:
\begin{enumerate} \item If $\varphi_t$ is transitive then the foliations $\F^{ws}$ and $\F^{wu}$ are minimal.
\item If $\varphi_t$ is not transitive then there exists a decomposition of the non-wandering set, $\Omega(\varphi_t)=\Lambda_1 \cup \ldots \cup \Lambda_K$, in disjoint \emph{basic pieces} $\Lambda_i$ that are compact, $\varphi_t$-invariant and such that $\varphi_t|_{\Lambda_i}:\Lambda_i\to \Lambda_i$ is transitive. Moreover, some of them, $\Lambda_1, \ldots, \Lambda_k$, are \emph{attracting  basic pieces} such that its whole basin $\F^{ws}(\Lambda_1)\cup \ldots \cup \F^{ws}(\Lambda_k)$ is an open and dense subset of $M$.
\end{enumerate} 
\end{remark}

\begin{definition}[Discretized Anosov flow]
A partially hyperbolic diffeomorphism $f:M\to M$ with $\dim(E^c)=1$ is a \emph{discretized Anosov flow} if there exist a topological Anosov flow $\varphi_t:M\to M$ and a continuous function $\tau:M\to \mathbb{R}_{>0}$ such that $$f(x)=\varphi_{\tau(x)}(x)$$ for every $x$ in $M$.
\end{definition}

\begin{remark} For every $f$ that is the time 1 map of an Anosov flow there exists a $C^1$-neighborhood $\mathcal{U}_f$ such that every $g$ in $\mathcal{U}_f$ is a discretized Anosov flow (see for example \cite{BG1}). This is a consequence of $(f,\W^c)$ being plaque expansive.
\end{remark}

\begin{remark} In \cite{BFFP} discretized Anosov flows are shown to be, modulo finite iterate, the unique type of partially hyperbolic diffeomorphism homotopic to the identity in many 3-manifold. This is the case for Seifert 3-manifolds (\cite[Theorem A]{BFFP}) and for dynamically coherent ones in hyperbolic 3-manifolds (\cite[Theorem B]{BFFP}).
\end{remark}

The following is contained in \cite[Proposition G.1]{BFFP} in the case $\dim(M)=3$ but is valid in any dimension:

\begin{prop}\label{propdafdc} Let $f(x)=\varphi_{\tau(x)}(x)$ be a discretized Anosov flow. Then $f$ is dynamically coherent, the weak-stable and weak-unstable foliations $\F^{ws}$ and $\F^{wu}$ of $\varphi_t$ are center-stable and center-unstable foliations $\W^{cs}$ and $\W^{cu}$ for $f$ tangent to $E^s \oplus E^c$ and $E^c \oplus E^u$, respectively, and the flow lines of $\varphi_t$ form a center foliation $\W^c$ for $f$ that is tangent to $E^c$.
\end{prop} 

Let us finish this subsection by showing examples of discretized Anosov flows with arbitrary number of quasi-attractors, even infinitely many (countable or uncountable).

We say that a flow $\varphi_t:X\to X$ is a \emph{suspension flow} if there exists a homeomorphism $g:Y\to Y$ such that the flow $\varphi_t$ is the projection of the flow in $Y\times \mathbb{R}$ generated by the vector field $\frac{\partial}{\partial t}=(0,1)$ into the quotient $X=Y\times \mathbb{R}/_\sim$ given by $(y,t+1)\sim (g(y),1)$. Notice that for a suspension flow the space $X$ has the structure of a bundle over the circle $S^1$ with fibers homeomorphic to $Y$. Moreover, the flow $\varphi_t$ takes fibers to fibers and the time 1 map $\varphi_1:X\to X$ leaves invariant each fiber (it projects as the identity on the base) and acts on each of them as the map $g$.

\begin{example}[Perturbing the time 1 map of an Anosov's suspension]\label{exSusp} Let us consider $\varphi_t:M\to M$ to be the suspension of an Anosov diffeomorphism $g:N\to N$.

We can perturb the time 1 map of $\varphi_t$ in order to get a partially hyperbolic map $f$ that still preserves fibers but acts like a Morse-Smale in the base. For example, taking coordinates $x=(y,t)$, we can consider $f$ explicitly as
$f(y,t)=\varphi_{\tau(y,t)}(y,t)$ with $\tau(y,t)=1+\alpha\sin(2\pi k t)$ for any $\alpha\in (0,1)$. In this case, $f$ has $k$ proper attractors and $k$ proper repellers.

Further, for a suitable $1$-periodic map $h:\mathbb{R}\to (-1,1)$, the discretization $\tau(y,t)=1+h(t)$ can produce infinite number of quasi-attractors and quasi-repellers. It is sufficient for $h$ to have infinite zeros (countably or uncountably), each of them accumulated by positive and negative values.
\end{example}

\begin{example}(Perturbing the time 1 map of the Franks-Williams's example)\label{exFW}

Consider $\varphi_t:M\to M$ the Franks-Williams's example of a non-transitive Anosov flow (\cite{FW}). Let $\Lambda$ be the unique basic attracting piece for $\varphi_t$.

In this particular example the flow $\varphi_t|_\Lambda:\Lambda \to \Lambda$ is conjugate to the suspension of a derived from Anosov map $g:\mathbb{T}^2 \to \mathbb{T}^2$ restricted to its unique attractor $Y\subset \mathbb{T}^2$. In particular, the time 1 map $\varphi_1|_\Lambda$ leaves invariant the fibers of this bundle structure acting on each of them as $g|_Y:Y\to Y$. Notice that each fiber is already a minimal unstable lamination for $\varphi_1$.

Moreover, as the construction of $\varphi_t$ involves performing a surgery far from the suspension of $g|_Y:Y\to Y$, in fact, there exists a neighborhood $V$ of $Y$ with $g(V)\subset V$ such that $\varphi_t$ is conjugate to the suspension of $g|_V:V\to V$ in a neighborhood $U$ of $\Lambda$. 

So we can perturb $\varphi_1$ in a smaller bundle neighborhood $U'\subset U$ of $\Lambda$ to obtain a discretization $f(x)=\varphi_{\tau(x)}(x)$ that acts with arbitrary number of quasi-attractors in the base and leaves unchanged the dynamics of $\varphi_1$ outside $U$. This can be done as in the previous example by taking $\tau(y,t)=1+h(t)$ for a suitable $1$-periodic diffeomorphism $h:\mathbb{R}\to \mathbb{R}$ in the neighborhood $U'$ and glueing it with the constant $\tau=1$ outside $U$.

This construction produces arbitrary number of quasi-attractors for $f$, each one of them homeomorphic to $Y$. On the fibered neighborhood $U'$ of $\Lambda$ the map $f$ acts as $h$ in the base $S^1$ and as $g$ on the fibers near $\Lambda$.
\end{example}

\subsection{Quasi-isometrically action on $\W^c$}\label{prelim3} In this subsection we present a more technical definition that will allow us to state and prove our results in its general form in Propositions \ref{prop1} and \ref{prop2}.

Let $f:M\to M$ be a dynamically coherent partially hyperbolic diffeomorphism with $\dim(E^c)=1$. 

\begin{definition} We say that $f$ acts \emph{quasi-isometrically in the future} on $\W^c$ if there exist some constants $l,L>0$ such that \begin{equation}\label{defqi} f^n(\W^c_l(x))\subset \W^c_L(f^n(x))\end{equation} for every $x$ in $M$ and $n\geq 0$. We analogously define \emph{quasi-isometrically in the past} if (\ref{defqi}) is verified for every $n\leq 0$. If $f$ acts quasi-isometrically in the past and future we say that $f$ acts \emph{quasi-isometrically} on $\W^c$.
\end{definition}

\begin{remark}\label{rem3.1} Every discretized Anosov flow $f(x)=\varphi_{\tau(x)}(x)$ acts quasi-isometrically on $\W^c$ with constants $l=\min\tau$ and $L=\max \tau$. 
\end{remark}

\begin{remark}\label{rem3.2} Every one-dimensional center partially hyperbolic skew-product acts quasi-isometrically on $\W^c$. This is because the length of center leaves is uniformly bounded in $M$ so it is enough to choose $L$ any constant larger than this bound. In fact, the same is valid if we ask $\W^c$ to be \emph{uniformly compact}, that is, if all leaves of $\W^c$ are compact and the leaf length function $x\mapsto \length(\W^c(x))$ is bounded in $M$.
\end{remark}

\begin{notation}
Given $x$ in $M$ and $y$ in $\W^c(x)$ we will denote by $[x,y]_c$ a center segment from $x$ to $y$ inside $\W^c(x)$. Notice that a priori there exist two alternatives for $[x,y]_c$ if $\W^c(x)$ is a circle. Later we will work with $\W^c$ oriented and this ambiguity will disappear.
\end{notation}

The following lemma and its corollary will be used later in the following sections. We say that a curve $x^u:[0,1]\to M$ is an \emph{unstable curve} if it is everywhere tangent to the bundle $E^u$.

\begin{lemma}\label{lemma2.1} Suppose $f$ acts quasi-isometrically in the past on $\W^c$. Let $[x,y]_c$ be a center segment in $M$ and $x^u:[0,1]\to M$ be an unstable curve such that $x^u(0)=x$. Then there exist a unique unstable curve $y^u:[0,1]\to M$ and unique center segments $[x^u(t),y^u(t)]_c$ varying continuously with $t\in [0,1]$ in the Hausdorff topology and verifying $[x^u(0),y^u(0)]_c=[x,y]_c$.
\end{lemma}

In the setting of the previous lemma, we will say that $\{[x^u(t),y^u(t)]_c\}_{t\in [0,1]}$ is the \emph{transport by unstable holonomy} of the center segment $[x,y]_c$ along the unstable curve $x^u$. 

\begin{proof}
As $f$ acts quasi-isometrically in the past on $\W^c$ there exists $R>0$ such that the length of $f^{-n}([x,y]_c)$ is less than $R$ for every $n\geq 0$.

As $f$ is dynamically coherent, leaves of $\W^{cu}$ are subfoliated by leaves of $\W^c$ and $\W^u$ having local product structure. Then, as $\W^c$ is tangent to a continuous bundle, there exist small constants $\epsilon,\delta>0$ such that for every center segment $[x',y']_c$ of length less than $R$ and every point $x''\in \W^u_\delta(x')$ there exists a unique center segment $[x'',y'']_c$ with $y''\in \W^u_\epsilon(y')$ such that $[x'',y'']_c$ is contained in an unstable $\epsilon$-neighborhood of $[x',y']_c$ of the form $\bigcup_{z\in [x',y']_c}\W^u_\epsilon(z)$. In other words, the transport by unstable holonomy for center segments of length at most $R$ is well defined along any unstable curve of a certain small length $\delta>0$.

Since $f$ contracts unstable distances in the past, then $f^{-n_0}(x^u)$ will have length less that $\delta$ for some $n_0\geq 0$. Then the transport by unstable holonomy of $[f^{-n_0}(x),f^{-n_0}(y)]_c$ along $f^{-n_0}(x^u)$ is well defined and, iterating it $n_0$ to the future, the transport by unstable holonomy of $[x,y]_c$ along $x^u$ is also well defined.
\end{proof}

A priori for a dynamically coherent partially hyperbolic diffeomorphism, for a given a center-unstable leaf $\W^{cu}(x)$ it might be the case that the set $\W^u(\W^c(x)):=\bigcup_{y\in \W^c(x)}\W^u(y)$ is contained but does not coincide with $\W^{cu}(x)$. Analogously it might be the case that $\W^c(\W^u(x)):=\bigcup_{y\in \W^u(x)}\W^c(y)$ does not coincide with $\W^{cu}(x)$. This does not happen if $f$ acts quasi-isometrically in the past on $\W^c$:

\begin{cor}\label{cor2.3}
If $f$ acts quasi-isometrically in the past on $\W^c$ then  \newline $\W^u(\W^c(x))=\W^c(\W^u(x))=\W^{cu}(x)$ for every $x$ in $M$.
\end{cor}

\begin{proof}
Let $\W^{cu}(x)$ be a center-unstable leaf in $M$ and $x'$ be any point in $\W^{cu}(x)$. We can join $x$ and $x'$ by a concatenation of finite $c$ and $u$ arcs. More precisely, there exist arcs $[x,x_1]_c$, $[x_1,x_2]_u$, \ldots, $[x_{n-2},x_{n-1}]_u$, $[x_{n-1},x']_c$ that join $x$ to $x'$. 

We want to see that $x$ and $x'$ can be joined by a $cu$-path, that is, a concatenation of one center segment and one unstable segment. Indeed, by a transport by unstable holonomy we can replace any consecutive $uc$-path $[x_{i-2},x_{i-1}]_u$, $[x_{i-1},x_i]_c$ by a $cu$-path $[x_{i-2},x'_{i-1}]_c$, $[x'_{i-1},x_i]_u$ joining the same starting point $x_{i-2}$ with the same ending point $x_i$. By applying this process successively we obtain a $cu$-path from $x$ to $x'$. A $uc$-path from $x$ to $x'$ can be constructed analogously. The conclusion of the Lemma follows.
\end{proof}

\begin{remark} Lemma \ref{lemma2.1} and Corollary \ref{cor2.3} admit obvious analogous statements for center-stable leaves if $f$ acts quasi-isometrically in the future on $\W^c$.
\end{remark}

\section{Proof of Theorem \ref{thmA} and Theorem \ref{propB}}\label{sectionproofs}

Theorem \ref{thmA} and Theorem \ref{propB} will be a consequence of the following two more general statements. Together they can be seen as an obstruction to the existence of more that one minimal unstable lamination (or more than one attracting region) for certain partially hyperbolic systems with one-dimensional center.

\begin{prop}\label{prop1} 
Let $f:M\to M$ be a dynamically coherent partially hyperbolic diffeomorphism with $\dim(E^c)=1$. Suppose that $\W^{cs}$ is minimal and $f$ acts quasi-isometrically in the future on $\W^c$. Then there exists $L>0$ such that every minimal unstable lamination $A$ verifies that $$\W^c_L(x)\cap A\neq \emptyset$$ for every $x\in M$. In particular, $A$ intersects every leaf of $\W^c$.
\end{prop}

\begin{remark} In Proposition \ref{prop1} the hypothesis `$\W^{cs}$ minimal' can be replaced by `$f$ chain-transitive' or `$f$ volume preserving' since the latters imply the former (see for example \cite[Lemma 1.1]{BW}).
\end{remark}

We say that a one-dimensional center manifold $\W^c$ admits a \emph{global section} if there exists a codimension one closed submanifold $N\subset M$ transverse to the leaves of $\W^c$ such that $\W^c_L(x)\cap N\neq \emptyset$ for every $x\in M$ and some constant $L>0$.

\begin{prop}\label{prop2}
Let $f:M\to M$ be a dynamically coherent partially hyperbolic diffeomorphism with $\dim(E^c)=1$. Suppose that the foliation $\W^c$ is orientable and that there exists $L>0$ such that $\W^c_L(x)\cap A\neq \emptyset$ for every minimal unstable lamination $A$. 

If $M$ has more than one minimal unstable lamination then $\W^c$ admits a global section.
\end{prop}

\begin{remark} Notice that Proposition \ref{prop2} does not include the hypothesis `$f$ acts quasi-isometrically on $\W^c$'. In fact, it is derived as a consequence along the proof (see Lemma 5.3.).
\end{remark}

\begin{proof}[Proof of Theorem \ref{thmA} assuming Propositions \ref{prop1} and \ref{prop2}]

Let $f(x)=\varphi_{\tau(x)}(x)$ be a discretized Anosov flow such that $\varphi_t:M\to M$ is transitive and not orbit equivalent to a suspension.

As pointed out in Proposition \ref{propdafdc} and Remark \ref{rmkAF}, the partially hyperbolic diffeomorphism $f$ is dynamically coherent, the foliation $\W^c$ coincides with the flow lines of $\varphi_t$ (in particular, it is oriented) and the foliation $\W^{cs}$ is minimal. Moreover, $f$ acts quasi-isometrically on $\W^c$ as pointed out in Remark \ref{rem3.1}.

Now, combining Propositions \ref{prop1} and \ref{prop2}, we obtain that $f$ can not admit more than one minimal unstable lamination, otherwise $\varphi_t$ would have a global section and then it would be orbit equivalent to a suspension flow. 
\end{proof}

Recall that a one-dimensional center manifold $\W^c$ is \emph{uniformly compact} if all leaves of $\W^c$ are compact and the leaf length function $x\mapsto \length(\W^c(x))$ is bounded in $M$. When $f$ is a skew-product, clearly $\W^c$ is uniformly compact (in fact, the leaf length function is continuous). We will prove Theorem \ref{propB} in its more general version for the case when $\W^c$ is an $f$-invariant uniformly compact foliation.

\begin{proof}[Proof of Theorem \ref{propB} assuming Propositions \ref{prop1} and \ref{prop2}]

Let $f:M\to M$ be a partially hyperbolic diffeomorphism with $\dim(E^c)=1$ admitting an $f$-invariant uniformly compact center foliation $\W^c$ such that the induced dynamics in the space of center leaves, $F:M/\W^c\to M/\W^c$, is transitive. Suppose that $M$ admits more that one minimal unstable laminations. We are going to see that under this hypothesis $(M,\W^c)$ has to be a virtually trivial bundle.

From \cite[Theorem 1]{BoBo} $f$ is dynamically coherent with center-stable and center-unstable foliations $\W^{cs}$ and $\W^{cu}$, respectively, such that $\W^c=\W^{cs}\cap \W^{cu}$. As $F:M/\W^c\to M/\W^c$ is transitive the foliation $\W^{cs}$ has to be minimal, otherwise a proper minimal set for $\W^{cs}$ would project to $M/\W^c$ into a proper repeller for $F$. 

Furthermore, as the length of center leaves is bounded, then $f$ automatically acts quasi-isometrically on $\W^c$ as pointed out in Remark \ref{rem3.2}.

Suppose first that $\W^c$ is orientable. Now by combining Propositions \ref{prop1} and \ref{prop2} we obtain that $\W^c$ admits a global section.

Let us denote the global section of $\W^c$ as $N\subset M$. Let $\alpha:N\to N$ denote the first return map of $\W^c$ to $N$, modulo fixing an orientation for $\W^c$. 

For every $x\in N$ let $k(x)\in \mathbb{Z}^+$ be the smallest positive integer such that $\alpha^{k(x)}(x)=x$. As $x\mapsto \length(\W^c(x))$ is bounded in $M$ there exists some constant $k \in \mathbb{Z}^+$ such that $k(x)\leq k$ for every $x$. By taking $K=k!$ we obtain that $\alpha^K=\id$. 

Let us consider a metric in $M$ such that every center segment $[x,\alpha(x)]_c$ is of length $\frac{1}{K}$ and let $\phi^c:M\to M$ denote the flow by arc-length whose flow lines are the leaves of $\W^c$. Then the map $p:N\times S^1 \to M$ given by $(x,\theta)\mapsto \phi^c_\theta(x)$ is a $K:1$ covering map sending circles of the form $\{\cdot\}\times S^1$ to leaves of the foliation $\W^c$. We conclude that $(M,\W^c)$ is a virtually trivial bundle.

In the case $\W^c$ is not orientable we can argue as above after taking an orientable double cover for $\W^c$. Indeed, we can lift $f$, $\W^c$ and all the minimal unstable laminations to an orientable double cover $\tilde{M}$. The quasi-isometrically action of $f$ on $\W^c$ remains valid on the lifted dynamics. We claim that the minimality of $\W^{cs}$ also remains valid on the lift. Indeed, if we suppose that the lift of $\W^{cs}$ is not minimal then there exist $\tilde{x}$ and $\tilde{x}'$ lifts of a point $x\in M$ such that $\overline{\W^{cs}(\tilde{x})}$ and $\overline{\W^{cs}(\tilde{x}')}$ are minimal proper subsets of $\tilde{M}$. Then $\tilde{M}$ coincides with the disjoint union $\overline{\W^{cs}(\tilde{x})}\cup\overline{\W^{cs}(\tilde{x}')}$ and we get to a contradiction. This proves the claim.

We obtain that the lifted dynamics verifies Propositions \ref{prop1} and \ref{prop2}. Then, as argued above, $\tilde{M}$ and the lift of $\W^c$ form a virtually trivial bundle. We conclude that $(M,\W^c)$ is also a virtually trivial bundle. 
\end{proof}

\section{Proof of Proposition \ref{prop1}}
Let $f:M\to M$ be as in the hypothesis of Proposition \ref{prop1}.

Recall that for every $r>0$ and $x\in M$ we denote by $\W^s_r(\W^c_r(x))$ the set $\bigcup_{y\in \W^c_r(x)} \W^s_r(y)$. As a consequence of $f$ acting quasi-isometrically in the future on $\W^c$ the Corollary \ref{cor2.3} gives us that $\bigcup_{r \geq 0} \W^s_r(\W^c_r(x))=\W^{cs}(x)$ for every $x$ in $M$.

\begin{lemma} There exists $R>0$ such that  $\W^s_R(\W^c_R(x))\cap \W^u(y)\neq \emptyset$ for every $x$ and $y$ in $M$.
\end{lemma} 
\begin{proof}
By contradiction, suppose there exist $R_n\xrightarrow{n}\infty$ and sequences $\{x_n\}_n$ and $\{y_n\}_n$ in $M$ such that $\W^s_{R_n}(\W^c_{R_n}(x_n))\cap \W^u(y_n)= \emptyset$ for every $n$. Then, as we are dealing with leaves of foliations tangent to continuous bundles, by taking accumulation points $x$ and $y$ of the sequences $\{x_n\}_n$ and $\{y_n\}_n$ we obtain that $\W^{cs}(x)\cap \W^u(y)=\emptyset$. This contradicts that $\W^{cs}(x)$ is dense in $M$.
\end{proof}

As $f$ acts quasi-isometrically in the future on $\W^c$ there exists $L>0$ such that $f^n(\W^c_R(x))$ is contained in $\W^c_L(f^n(x))$ for every $n\geq 0$ and $x\in M$.

Proposition  \ref{prop1} is a direct consequence of the following lemma.

\begin{lemma} For every $x$ and $y$ in $M$ we have that $\W^c_L(x) \cap \overline{ \W^u(y)}\neq \emptyset$.
\end{lemma}
\begin{proof}
Let us fix $x$ and $y$ arbitrary points in $M$. For every $n\geq 0$ we have that $\W^s_R(\W^c_R(f^{-n}(x)))\cap \W^u(f^{-n}(y))\neq \emptyset$. Then, as $f$ contracts distances uniformly to the future inside stable leaves, there exists $r_n\xrightarrow{n} 0$ such that the image of $\W^s_R(\W^c_R(f^{-n}(x)))$ by $f^n$ is contained in $\W^s_{r_n}(\W^c_L(x))$. We obtain that $\W^s_{r_n}(\W^c_L(x)) \cap \W^u(y)\neq \emptyset$ for every $n\geq 0$ and then $\W^c_L(x) \cap \overline{ \W^u(y)}\neq \emptyset$.
\end{proof}

\section{Proof of Proposition \ref{prop2}}

From now on let $f:M\to M$ and $L>0$ be as in the hypothesis of Proposition \ref{prop2} and suppose that there exist $A$ and $A'$ different minimal unstable laminations in $M$. We will see that under this hypothesis $\W^c$ has to admit a global section.

\subsection{The sets $(A,A')_c$ and $(A',A)_c$}
The goal of this subsection is to prove that the sets $(A,A')_c$ and $(A',A)_c$ defined below are disjoint open subsets of $M$ that ``separates'' the disjoint and closed subsets $[A]_c$ and $[A']_c$ (see Proposition \ref{prop22}). 

Let us fix from now on an orientation for $\W^c$ and denote $\phi^c:M\times \mathbb{R} \to M$ a non-singular flow that parameterizes the leaves of $\W^c$. 

\begin{notation} For $x$ and $y$ in the same center leaf we will say that $x\leq y$ if $y= \phi^c_t(x)$ for some $t\geq 0$. If this is the case, let $(x,y)_c$ and $[x,y]_c$ denote the open and closed center segments from $x$ to $y$.
\end{notation}

Let us define the sets: 
\begin{equation*}
  \label{eq:t}
  \begin{aligned}
    [A]_c&=\bigcup \{[x,y]_c: x\in A, y\in A,[x,y]_c\cap A'=\emptyset\},\\        
    (A,A')_c&=\bigcup \{(x,y)_c: x\in A, y\in A',(x,y)_c\cap (A\cup A')=\emptyset\}.
  \end{aligned}
\end{equation*}
Notice that the center segments in the definition of $[A]_c$ may be singletons. We define analogously the sets $[A']_c$ and $(A',A)_c$. By an abuse of notation, we will consider this sets both as subsets of $M$ and as an abstract collection center segments.

The following remark is a direct consequence from the definitions.

\begin{remark}
The manifold $M$ is equal to the disjoint union $[A]_c\cup (A,A')_c \cup [A']_c \cup (A',A)_c$.
\end{remark}

Let us first point out that:

\begin{lemma}\label{qiprop2}
The map $f$ acts quasi-isometrically on $\W^c$. 
\end{lemma}

\begin{proof}
Let $d>0$ be the distance between the disjoint minimal unstable laminations $A$ and $A'$. 

We claim that, as every center segment of length $2L>0$ intersects every minimal unstable lamination, then $f^n(\W^c_d(x))$ can not have length larger than $2L$ for any $x\in M$ and $n\in \mathbb{Z}$. 

By contradiction, if the length of $f^n(\W^c_d(x)))$ is larger that $2L$ for some $x\in M$ and $n\in \mathbb{Z}$ then $f^n(\W^c_d(x))$ intersects both minimal unstable laminations $f^n(A)$ and $f^n(A')$. Then $\W^c_d(x)$ has to intersect both $A$ and $A'$. This gives us a contradiction and the claim is proved. We obtain that $f$ is acts quasi-isometrically on $\W^c$ with constants $d,2L>0$.
\end{proof}

As a consequence of the previous lemma the equality $\W^u(\W^c(x))=\W^{cu}(x)$ is verified for every $x\in M$ (Corollary \ref{cor2.3}) and we will be able to make `long' transports by unstable holonomy of any center segment (Lemma \ref{lemma2.1}).

For every $x\in A$ let us define $S(x)$ to be the first point of $A'$ in $\W^c(x)$ in the direction of the flow $\phi^c$. That is, $S(x)$ is such that $(x,S(x))_c$ is a center segment in $(A,A')_c$. Let us define $l_S(x)$ as the length of the arc $[x,S(x)]_c$.

\begin{lemma}
The function $l_S:A\to \mathbb{R}$ is lower semicontinuous and continuous in a residual subset of $A$.
\end{lemma}
\begin{proof}
Since $l_S$ is bounded from above by the constant $2L>0$ then, for every sequence $\{x_n\}_n\subset A$ that converges to a point $x$ in $A$, any accumulation point $y$ of $S(x_n)$ lies in $\W^c(x)$. Since $A'$ is closed, $y$ is a point in $A'$. Then $[x,S(x)]_c$ has to be contained in $[x,y]_c$ for any such an accumulation point $y$. This implies that $l_S(x)\leq \liminf_n l_S(x_n)$ and we obtain that $l_S$ is lower semicontinuous.

To see that $A$ contains a residual subset of continuity points for $l_S$ we can consider the sets $F_m=\{x\in A:\exists \text{ } x_n \xrightarrow{n}x \text{ s.t. } \liminf_n l_S(x_n) \geq l_S(x) +\frac{1}{m}\}$ for every $m$ in $\mathbb{Z}^+$. The set of continuity points of $l_S$ coincides with $A\setminus \bigcup_m F_m$. It is direct to prove that each $F_m$ is a closed nowhere dense subset of $A$. Then $A\setminus \bigcup_m F_m$ is a residual set in $A$ by Baire category theorem.


\end{proof}

For a continuity point $x$ of $l_S$ every sequence $\{x_n\}_n\subset A$ converging to $x$ verifies that the center segments $[x_n,S(x_n)]_c$ converges in the Hausdorff topology to $[x,S(x)]_c$. For a discontinuity point this is not the case, however, we will see in the following lemma that the failure of continuity is not so bad. We will use that the behavior of $S$ near a continuity point can be extended by unstable holonomy to any point of $A$ thanks to Lemma \ref{lemma2.1}:

\begin{lemma}\label{lemma21}
Let $\{x_n\}_n\subset A$ be a sequence converging to a point $x\in A$. Up to taking a subsequence, suppose that $\{S(x_n)\}_n$ converges to a point $y\in A'$. Then $y$ lies in $\W^c(x)$, the center segments $[x_n,S(x_n)]_c$ converge in the Hausdorff topology to $[x,y]_c$ and $[S(x),y]_c$ is a center segment (possibly degenerate to a point) contained in $[A']_c$.
\end{lemma}
\begin{proof}
We claim first that the lemma is true for every $x\in A$ in a neighborhood of a continuity point of $l_S$.

Indeed, let $z\in A$ be a continuity point of $l_S$ and consider $U_{S(z)}$ a small neighborhood of $S(z)$ at a positive distance from $A$. We can suppose that $U_{S(z)}$ is a foliation box of $\W^c$, that is, that $U_{S(z)}$ is the image of a homeomorphism $h:D\times [0,1] \to U_{S(z)}$ such that $D$ is a compact disc of dimension $\dim(M)-1$ and $h(\{x\}\times [0,1])$ is a center segment for every $x\in D$. Let us denote by $D_1$ the disc $h(D\times \{1\})$.

Since $z$ is a continuity point of $l_S$ we can consider $\delta>0$ such that for every $x\in A\cap B_\delta(z)$ we have that $S(x)$ lies in the interior of $U_{S(z)}$. In particular, the center segment $[x,S(x)]_c$ does not cross the disc $D_1$.

If $\{x_n\}_n\subset A$ is a sequence converging to a point $x\in A\cap B_\delta(z)$, then any accumulation point $y$ of $\{S(x_n)\}_n$ has to lie in $U_{S(z)}$. Up to a subsequence, let us assume that $S(x_n)\xrightarrow{\raisebox{-0.2ex}[0ex][0ex]{\scriptsize{$n$}}}y$. Then, as each $[x_n,S(x_n)]_c$ does not intersect $D_1$, the segments $[x_n,S(x_n)]_c$ need to converge in the Hausdorff topology to $[x,y]_c$ and the whole segment $[S(x),y]_c$ has to be contained $U_{S(z)}$ (see Figure \ref{fig0}). As $S(x)$ and $y$ are in $A'$ and $U_{S(z)}$ is disjoint from $A$ we conclude that $[S(x),y]_c$ is a center segment in $[A']_c$. This proves the first claim.
\begin{figure}[htb]
\def\svgwidth{0.8\textwidth}
\begin{center} 
{\scriptsize
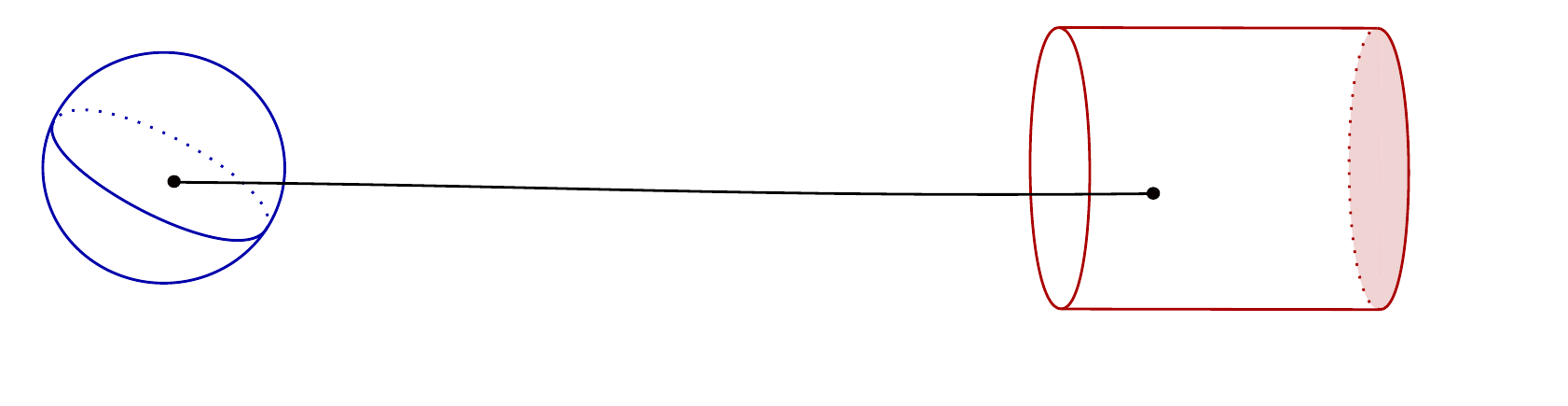
}
\caption{}\label{fig0} 
\end{center}
\end{figure}

Let us see now that the lemma is true for every point in $\hat{x} \in A$. We will use as an auxiliary construction a continuity point $z$ for $l_S$ and the neighborhoods $B_\delta(z)$ and $U_{S(z)}$ as in the previous claim.

Let $\{\hat{x}_n\}_n\subset A$ be a sequence converging to the point $\hat{x}\in A$. Suppose, up to taking a subsequence, that $S(\hat{x}_n)$ converges to a point $\hat{y}$. 
As $A$ is $\W^u$-minimal and $z\in A$ we can consider $x\in \W^u(\hat{x})\cap B_\delta(z)$ and $x^u:[0,1]\to M$ an unstable arc such that $x^u(0)=x$ and $x^u(1)=\hat{x}$. We can consider also unstable arcs $\{x^u_n:[0,1]\to M\}_n$ converging uniformly to $x^u$ such that $x^u_n(0)=x_n$ lies in $\W^u(\hat{x}_n)\cap B_\delta(z)$ and $x^u_n(1)=\hat{x}_n$.

\begin{figure}[htb]
\def\svgwidth{0.8\textwidth}
\begin{center} 
{\scriptsize
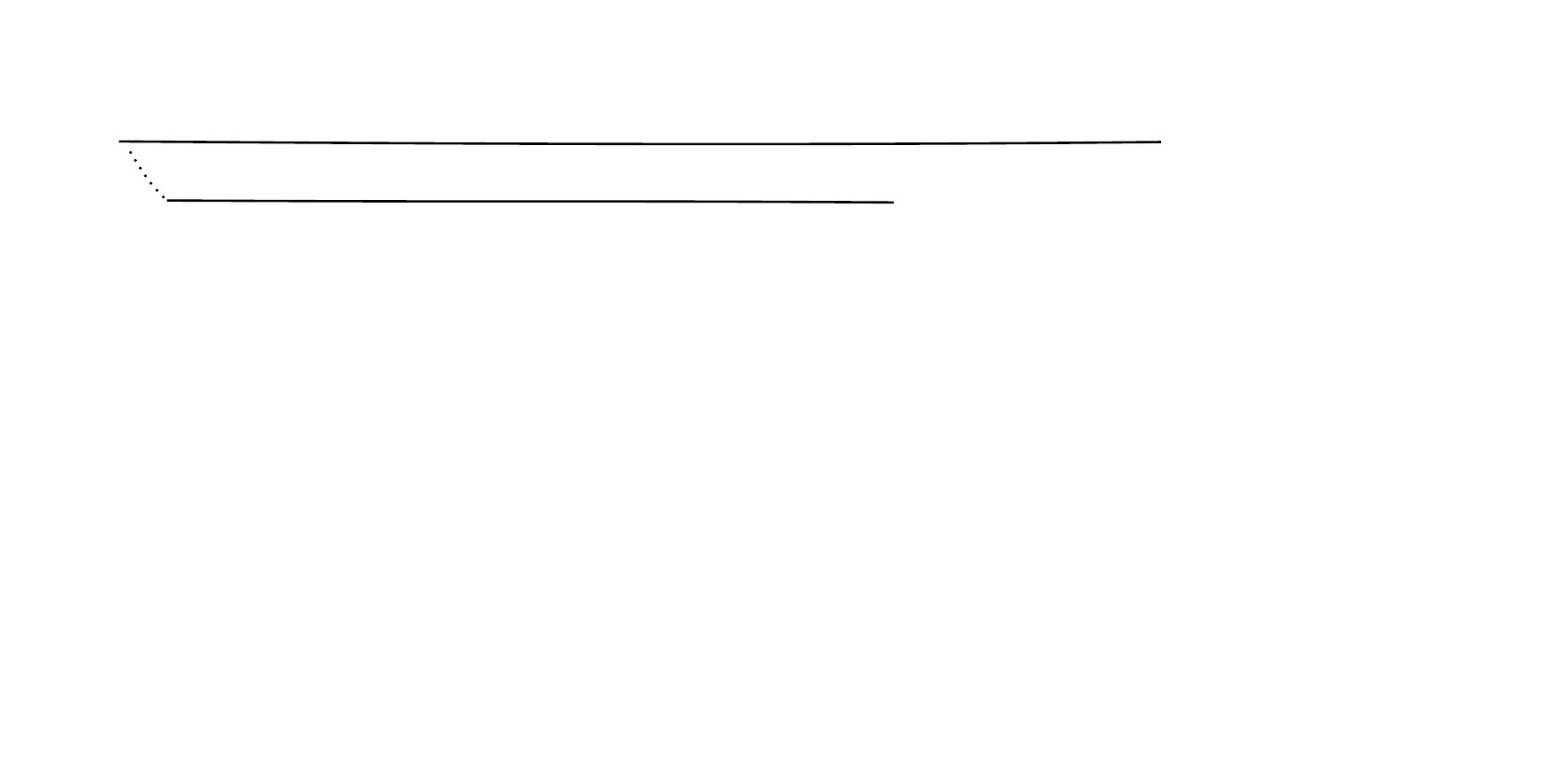
}
\caption{}\label{fig1} 
\end{center}
\end{figure}

Let us denote $y_n=S(x_n)$ for every $n$. Notice that $y_n\in U_{S(z)}$ since $x_n\in B_\delta(z)$. As the sequence $x_n$ converges to $x$ then by the first claim, up to taking a subsequence, $[x_n,y_n]_c$ converges in the Hausdorff topology to a center segment $[x,y]_c$ such that $[S(x),y]_c$ is in $[A']_c$.

Consider now $\{[x^u(t),y^u(t)]_c\}_{t\in [0,1]}$ the transport by unstable holonomy of $[x,y]_c$ along $x^u$ such that $[x^u(0),y^u(0)]_c=[x,y]_c$ (see Lemma \ref{lemma2.1}). Consider also $\{[x^u_n(t),y^u_n(t)]_c\}_{t\in [0,1]}$ the transport by unstable holonomy of $[x^u_n(0),y^u_n(0)]_c=[x_n,y_n]_c$ along $x^u_n$ for every $n$. Notice that, as $S(x_n)=y_n$ and $A$ and $A'$ are $\W^u$-invariant, then $S(x^u_n(t))=y^u_n(t)$ for every $t\in [0,1]$.

The foliations $\W^c$ and $\W^u$ have $C^1$ leaves tangent to continuous subbundles of $M$. So as $x^u_n$ converges uniformly to $x^u$ we have that $[x^u_n(t),y^u_n(t)]_c$ has to converge in the Hausdorff topology to $[x^u(t),y^u(t)]_c$ for every $t\in [0,1]$. In particular, the sequence $[x^u_n(1),y^u_n(1)]_c=[\hat{x}_n,S(\hat{x}_n)]_c$ needs to converge to $[x^u(1),y^u(1)]_c$. As $x^u(1)=\hat{x}$ and $S(\hat{x}_n)\xrightarrow{\raisebox{-0.2ex}[0ex][0ex]{\scriptsize{$n$}}} \hat{y}$ we obtain that $[x^u(1),y^u(1)]_c=[\hat{x},\hat{y}]_c$. Then the sequence $[\hat{x}_n,S(\hat{x}_n)]_c$ converges to $[\hat{x},\hat{y}]_c$. 

Finally, from the first claim, $[S(x),y]_c=[S(x^u(0)),y^u(0)]_c$ is a center segment in $[A']_c$. This property is preserved by unstable holonomy since $A$ and $A'$ are $\W^u$-saturated so $[S(x^u(t)),y^u(t)]_c$ is a center segment in $[A']_c$ for every $t\in [0,1]$. We conclude that $[S(x^u(1)),y^u(1)]_c=[S(\hat{x}),\hat{y}]_c$ needs to be a center segment in $[A']_c$ (see Figure \ref{fig1}) and this ends the proof of the lemma.
\end{proof}

We are now able to prove:

\begin{prop}\label{prop22} The sets $(A,A')_c$ and $(A',A)_c$ are disjoint open subsets of $M$. The sets $[A]_c$ and $[A']_c$ are disjoint closed subsets of $M$. 
\end{prop}
\begin{proof}
The sets $(A,A')_c$ and $(A',A)_c$ are disjoint by definition. For proving $(A,A')_c$ open let us see that its complement, $[A]\cup (A,A')_c \cup [A']_c$, is closed. The proof of $(A',A)_c$ open is analogous.

Let $\{v_n\}_n$ be a sequence in $[A]_c\cup (A,A')_c \cup [A']_c$ converging to a point $v$. The sequence $\{v_n\}_n$ lies infinitely many times in $[A]_c\cup (A,A')_c$ or $(A,A')_c\cup [A']_c$. Suppose without loss of generality that it is the former. So, up to a subsequence, there exist $x_n\in A$ and $y_n=S(x_n)\in A'$ such that $v_n$ lies in the center segment $[x_n,y_n]_c$ for every $n$.

Then by Lemma \ref{lemma21}, up to taking a converging subsequence such that $x_n \xrightarrow{n} x$ and $y_n \xrightarrow{n} y$, the sequence $[x_n,y_n]_c$ converges in the Hausdorff topology to the center segment $[x,y]_c$ and $[S(x),y]_c$ is in $[A']_c$. 

Then $[x,y]_c=[x,S(x)]_c\cup[S(x),y]_c$ is a center segment in $[A]\cup (A,A')_c \cup [A']_c$. As the limit point $v$ needs to lie in $[x,y]_c$ this proves that $[A]\cup (A,A')_c \cup [A']_c$ is closed.

The sets $[A]_c$ and $[A']_c$ are disjoint by definition. Let us see that $[A]_c$ is closed. The proof of $[A']_c$ closed is analogous.

Let $\{w_n\}_n$ be a sequence in $[A]_c$ converging to a point $w$. Suppose that each $w_n$ is contained in a segment $[x_n,z_n]_c$ in $[A]_c$ and consider $y_n=S(x_n)$ for every $n$. Then by Lemma \ref{lemma21}, up to taking a converging subsequence such that $x_n \xrightarrow{n} x$ and $y_n \xrightarrow{n} y$, the sequence $[x_n,y_n]_c$ converges in the Hausdorff topology to the center segment $[x,y]_c$ such that $[S(x),y]_c$ is in $[A']_c$.

Up to taking another subsequence if necessary the sequence $\{z_n\}_n\subset A$ converges to a point $z\in A$ contained in $[x,y]_c$. The sequence $[x_n,z_n]_c$ converges in the Hausdorff topology to $[x,z]_c$ so the point $w$ needs to lies in $[x,z]_c$ as it is the limit of points $w_n$ in $[x_n,z_n]_c$. Moreover, since $[S(x),y]_c\cap A=\emptyset$, then $[x,z]_c$ needs to be contained in $[x,S(x)]_c$. We deduce that $[x,z]_c$ is a center segment in $[A]_c$ containing $w$. This proves that $[A]_c$ is closed in $M$.
\end{proof}

Let us end this subsection with a small parenthesis:

\begin{remark}
Let us compare with the examples in \cite{BG2} of Axiom A discretized Anosov flows having a proper attractor $\Lambda$ and a proper repeller $\Lambda'$ such that $\W^c$ is not topologically conjugate to a suspension.

In these examples, $\Lambda$ and $\Lambda'$ are minimal unstable and stable laminations, respectively, and Proposition \ref{prop1} is verified: there exist $L>0$ such that $\W^c_L(x)$ intersects $\Lambda$ and $\Lambda'$ for every $x\in M$. 

If one is tempted to imitate the present proof with $\Lambda$ and $\Lambda'$ in the place of $A$ and $A'$, it fails at the following point: By considering analogously the sets $(\Lambda,\Lambda')_c$ and $(\Lambda',\Lambda)_c$ of center segments from $\Lambda$ to $\Lambda'$ and from $\Lambda'$ to $\Lambda$, respectively, the property that is not verified is that $(\Lambda,\Lambda')_c$ and $(\Lambda',\Lambda)_c$ are open. Indeed, there exist segments in $(\Lambda,\Lambda')_c$ accumulated by segments from  $(\Lambda',\Lambda)_c$, and vice versa. The basin of $\Lambda$ or $\Lambda'$ does not decomposes in two connected components, rather $\Lambda$ and $\Lambda'$ are geometrically intertwined in such a fashion that its basins have a unique connected component.
\end{remark}

\subsection{Construction of the global section}

We can conclude from the previous subsection that $M$ decomposes as the disjoint union $$M=[A]_c\cup (A,A')_c \cup [A']_c \cup (A',A)_c,$$ where $(A,A')_c$ and $(A',A)_c$ are open and $[A]_c$ and $[A']_c$ are closed in $M$. From Proposition \ref{prop1} there exists $L>0$ such that $\W^c_L(x)\cap A\neq \emptyset$ and $\W^c_L(x)\cap A'\neq \emptyset$ for every $x\in M$. Let us see in this subsection that all this is sufficient for proving that $\W^c$ has to admit a global section.

Consider $\theta:M \to [0,1]$ continuous such that $\theta^{-1}(0)=[A]_c$ and $\theta^{-1}(1)=[A']_c$. Define next $\rho:M \to S^1$ such that
$$
\rho(x)=
\left\{
	\begin{array}{ll}
		\frac{1}{2}\theta(x) \text{ (mod 1)}  & \mbox{if } x\in [A]_c\cup(A,A')_c\cup[A']_c \\
		1-\frac{1}{2}\theta(x) \text{ (mod 1)}& \mbox{if } x\in [A']_c\cup(A',A)_c\cup[A]_c
	\end{array}
\right.
$$

\begin{remark}
The function $\rho:M \to S^1$ is well-defined and continuous.
\end{remark}
\begin{proof}
If $x$ is a point belonging both to $[A]_c\cup(A,A')_c\cup[A']_c$ and $[A']_c\cup(A',A)_c\cup[A]_c$ then $x\in [A]_c=\theta^{-1}(0)$ or $x\in [A']_c=\theta^{-1}(1)$. In both cases, $\frac{1}{2}\theta(x)$ and $1-\frac{1}{2}\theta(x)$ take the same value (mod 1). We obtain that $\rho$ is well defined. 

Since $\rho$ is a continuous function restricted to each closed subset $[A]_c\cup(A,A')_c\cup[A']_c$ and $[A']_c\cup(A',A)_c\cup[A]_c$ (they are closed as they are the complement of $(A',A)_c$ and $(A,A')_c$, respectively), and since the union of both closed subsets is $M$, then $\rho$ is continuous.
\end{proof}

Recall that $\phi^c:M\times \mathbb{R} \to M$ denotes a flow whose flow lines are the leaves of $\W^c$. Let us assume that $\phi^c$ is parametrized by arclength. Let $p:\tilde{M}\to M$ be the universal cover of $M$ and $\tilde{\phi}^c:\tilde{M}\times \mathbb{R} \to \tilde{M}$ be the lift of $\phi^c$ to $\tilde{M}$. Consider $\tilde{\rho}:\tilde{M}\to \mathbb{R}$ to be a lift of $\rho:M\to S^1$, that is, such that $\pi \circ \tilde{\rho}=\rho \circ  p$.

As $A$ and $A'$ intersect every center segment of length $2L>0$ then for every $\tilde{x}$ in $\tilde{M}$:
\begin{equation}\label{eqschwartz} \tilde{\rho}\circ \tilde{\phi^c}(\tilde{x},4L)-\tilde{\rho} \circ \tilde{\phi^c}(\tilde{x},0)>1.\end{equation}
Notice that for a given $x$ in $M$ the difference considered in (\ref{eqschwartz}) is independent of the lift $\tilde{x}$ of $x$. Informally, it measures how much `winds around' $S^1$ the image by $\rho$ of the center segment $[x,x+4L]_c$.

Now an argument of Schwartzman (see \cite{Sc} and \cite{BG1}) allows us to conclude the proof of Proposition \ref{prop2}. We reproduce it for the sake of completeness.

\begin{prop} Let $M$ be a smooth manifold and $\phi:M\times \mathbb{R} \to M$ be a flow tangent to a continuous vector field $X_\phi$ in $M$ and satisfying (\ref{eqschwartz}) for a certain continuous function $\rho:M\to S^1$ and some constant $L>0$. Then $\phi$ admits a smooth global section.
\end{prop}

\begin{proof}
Let $p:\tilde{M}\to M$ be the universal cover of $M$. Consider $X_\psi$ a smooth vector field $C^0$-close to $X_\phi$ and $\mu:M\to S^1$ a smooth map $C^0$-close to $\rho:M\to S^1$. Let $\tilde{\psi}:\tilde{M}\to \tilde{M}$ be the lift to $\tilde{M}$ of the flow $\psi:M\to M$ tangent to $X_\psi$ and $\tilde{\mu}:\tilde{M}\to\mathbb{R}$ be such that $\tilde{\pi}\circ \tilde{\mu}=\mu \circ p$. Then, if $X_\psi$ and $\mu$ are close enough to $X_\phi$ and $\rho$, respectively,  we still have
$$\tilde{\mu}\circ \tilde{\psi}(x,4L)-\tilde{\mu}\circ \tilde{\psi}(x,0)>1,$$
for every $x\in M$.

Let us consider now the smooth map $\tilde{\lambda}:\tilde{M} \to \mathbb{R}$ given by
$$\tilde{\lambda}(\tilde{x})=\frac{1}{4L}\int_0^{4L}\tilde{\mu}\circ\tilde{\psi}(\tilde{x},t) dt. $$
We claim that $\tilde{\lambda}$ projects to a map $\lambda:M \to S^1$. Indeed, if $\tilde{x}$ and $\tilde{y}$ are two points in $\tilde{M}$ such that $x=p(\tilde{x})=p(\tilde{y})$ then there exists an integer $n$ such that $\tilde{\mu}(\tilde{y})=\tilde{\mu}(\tilde{x})+n$. Furthermore, $n$ satisfies that $\tilde{\mu}\circ \tilde{\psi}(\tilde{y},t)=\tilde{\mu}\circ \tilde{\psi}(\tilde{x},t)+n$ for every $t$. This implies that $\tilde{\lambda}(\tilde{y})=\tilde{\lambda}(\tilde{x})+n$. We deduce that $$\lambda(x):=\tilde{\lambda}(\tilde{x})\text{ (mod 1)}$$ is well defined independently of the lift $\tilde{x}$. This proves the claim.

Moreover, for any $\tilde{x}$ in $\tilde{M}$ we have:
\begin{equation}\label{lasteq}\frac{\partial}{\partial t} \tilde{\lambda}\circ \tilde{\psi}(\tilde{x},t)|_{t=0}=\frac{1}{4L} (\tilde{\mu}\circ \tilde{\psi}(\tilde{x},4L)-\tilde{\mu}\circ \tilde{\psi}(\tilde{x},0)) > \frac{1}{4L} >0.\end{equation}
This proves that $\lambda:M\to S^1$ is a submersion such that the orbits of $\psi$ are transverse to the fibers. We obtain that $N=\lambda^{-1}(0)$ is a submanifold of $M$ that is a global section for the flow $\psi$. 

Moreover, since (\ref{lasteq}) gives us a positive lower bound (which only depends on the a priori constant $L>0$) for the angle between the vector field $X_\psi$ and the fibers of $\lambda:M\to S^1$ then we can consider $X_\psi$ to be $C^0$-close enough to $X_\phi$ so that $\phi:M\to M$ is also transverse to the fibers and $N$ is a global section for $\phi$.
\end{proof}

\section{Proof of Theorem \ref{thmA'} and Corollary \ref{corA'}}\label{section6}

Consider from now on $f(x)=\varphi_{\tau(x)}(x)$ a discretized Anosov flow such that $\varphi_t:M\to M$ is a non-transitive topological Anosov flow.

\begin{proof}[Adapting the proof of Theorem \ref{thmA} to prove Theorem \ref{thmA'}]

Let $\Lambda$ be an attracting basic piece of $\varphi_t$. As the flow $\varphi_t|_\Lambda:\Lambda \to \Lambda$ is transitive one obtains that $\W^{cs}(x)\cap \Lambda$ is dense in $\Lambda$ for every $x$ in $\Lambda$, otherwise $\overline{\W^{cs}(x)\cap \Lambda}$ would be a proper repeller for $\varphi_t|_\Lambda$. So $\W^{cs}|_\Lambda$ is minimal in $\Lambda$. As $f$ acts quasi-isometrically on $\W^c$ then Proposition \ref{prop1} adapts identically and we get that there exists $L>0$ such that every  minimal unstable lamination $A$ in $\Lambda$ verifies that $\W^c_L(x)\cap A \neq \emptyset$ for every $x\in \Lambda$.

Suppose now that there exist two different minimal unstable laminations $A$ and $A'$ in $\Lambda$. We want to show that $\varphi_t|_\Lambda$ needs to be orbit equivalent to a suspension. This will conclude the proof of Theorem \ref{thmA'}.

We can analogously define the sets $[A]_c$, $(A,A')_c$, $[A']_c$ and $(A',A)_c$ as in the proof of Proposition \ref{prop2}. The proof that $(A,A')_c$ and $(A',A)_c$ are open and that $[A]_c$ and $[A']_c$ are closed in $\Lambda$ works analogously. This allows us to define a continuous function $\rho:\Lambda \to S^1$ such that
\begin{equation}\label{eqschwartz2} |\widetilde{\rho\circ \varphi}_{4L}(x)-\widetilde{\rho\circ \varphi}_{0}(x)|>1,\end{equation}
where $t\mapsto \widetilde{\rho\circ \varphi_t}(x):\mathbb{R}\to \mathbb{R}$ is any lift of $t\mapsto \rho\circ \varphi_t(x)$ for every $x\in \Lambda$.

We can extend now $\rho$ to a small open $\varphi_t$-forward invariant neighborhood $U$ of $\Lambda$ is the following way: We can cover $\Lambda$ by $B_{\delta_1}(x_1)\cup \ldots \cup B_{\delta_j}(x_j)$ such that $x_i\in \Lambda$ and $|\rho(x)-\rho(x_i)|<1/10$ for every $x\in \Lambda\cap B_{\delta_i}(x_i)$. By Tietze extension theorem we can extend $\rho|_{\Lambda\cap  B_{\delta_i}(x_i)}$ to $\rho_i:B_{\delta_i}(x_i)\to S^1$ such that we still have $|\rho_i(x)-\rho(x_i)|<1/10$ for every $x\in B_{\delta_i}(x_i)$. Then by taking a partition of unity $\{\tau_i:B_{\delta_i}(x_i)\to [0,1]\}_i$ subordinated to $\{B_{\delta_i}(x_i)\}_i$ the functions $\{\rho_i\}_i$ can be interpolated in order to obtain an extension of $\rho$ to $B_{\delta_1}(x_1)\cup \ldots \cup B_{\delta_j}(x_j)$. 
Finally, we can take $V\subset B_{\delta_1}(x_1)\cup \ldots \cup B_{\delta_j}(x_j)$ such that $\varphi_t(V)\subset B_{\delta_1}(x_1)\cup \ldots \cup B_{\delta_j}(x_j)$ for every $t\geq 0$ and then define $U=\bigcup_{t\geq 0}\varphi_t(V)$. 

This construction of $U$ gives us that (\ref{eqschwartz2}) continues to be valid for every $x\in U$. The argument of Schwartzman also works well  restricted to $U$: by taking smooth approximations $\mu$ and $X_\psi$ of $\rho$ and $\frac{\partial \varphi_t}{\partial t}|_{t=0}$ , respectively, we can define the function $\lambda:U\to S^1$ as $\lambda(x)=\frac{1}{4L}\int_0^L\widetilde{\mu \circ \psi_t}(x)dt \text{ (mod 1)}$ and obtain that $\frac{\partial}{\partial t} \widetilde{\lambda\circ \psi_t}(x)|_{t=0}> \frac{1}{4L} >0$ for every $x\in U$. Then $\lambda^{-1}(0)$ gives us a global forward section for $\varphi_t|_U$. This global forward section gives us a global section for $\varphi_t|_\Lambda$.
\end{proof}

\begin{proof}[Proof of Corollary \ref{corA'}]
Let $f(x)=\varphi_{\tau(x)}(x)$ be a discretized Anosov flow such that $\varphi_t$ is not transitive. Let $\Lambda_1$, \ldots, $\Lambda_k$ be the attracting basic pieces of $\varphi_t$ and suppose that $\varphi_t$ is not orbit equivalent to a suspension restricted to any of this pieces.

Recall that $\W^{cs}(\Lambda_1)\cup \ldots \cup \W^{cs}(\Lambda_k)$ is an open and dense $\W^{cs}$-saturated subset of $M$. As $f$ acts quasi-isometrically on $\W^c$ then $\W^{cs}(x)=\W^s(\W^c(x))$ for every $x$ in $M$ by Corollary \ref{cor2.3}.  As each $\Lambda_i$ is $\W^c$-saturated, then $\W^{cs}(\Lambda_1)\cup \ldots \cup \W^{cs}(\Lambda_k)$ coincides with $\W^s(\Lambda_1)\cup \ldots \cup \W^s(\Lambda_k)$.

We claim that there exists $R>0$ such that $$\W^u(x)\cap \big(\W^s_R(\Lambda_1)\cup \ldots \cup \W^s_R(\Lambda_k)\big)\neq \emptyset$$ for every $x$ in $M$. Indeed, let $V_1^u$,\ldots, $V_j^u$ be a finite collection of $\W^u$ -foliation boxes such that $\bigcup_i V^u_i=M$. For every $1\leq i \leq j$ there exist $R_i>0$ such that $\W^s_{R_i}(\Lambda_1)\cup \ldots \cup \W^s_{R_i}(\Lambda_k)$ intersects every $\W^u$-plaque in $V^u_i$. The claim follows from taking $R=\max\{R_1,\ldots, R_j\}$.

As a consequence of the previous claim we obtain that $$\overline{\W^u(x)} \cap \big(\Lambda_1 \cup \ldots \cup\Lambda_k\big)\neq \emptyset$$ for every $x$ in $M$. Indeed, as $\W^u(f^{-n}(x))$ intesercts $\big(\W^s_R(\Lambda_1)\cup \ldots \cup \W^s_R(\Lambda_k)\big)$ for every $n\geq 0$ then $\W^u(x)=f^n( \W^u(f^{-n}(x)))$ is at distance $0$ from $\Lambda_1 \cup \ldots \cup\Lambda_k$. We deduce that every minimal unstable lamination for $f$ intersects $\Lambda_1 \cup \ldots \cup\Lambda_k$. 


Moreover, as each attracting basic piece is compact and $\W^u$-saturated, then every minimal unstable lamination for $f$ has to be contained in one of the attracting basic pieces.

Finally, by Theorem \ref{thmA'}, each attracting basic piece $\Lambda_i$ contains a unique minimal unstable lamination. We conclude that $f$ admits exactly $k$ minimal unstable laminations and that each one of them is contained in one of the attracting basic pieces $\Lambda_1$, \ldots, $\Lambda_k$ of $\varphi_t$.
\end{proof}

Let us finish this section with a brief sketch on how to construct an Anosov flow in the hypothesis of Corollary \ref{corA'}:

\begin{example}
Let $S$ be a negatively curved hyperbolic closed surface. Let $\varphi_t:T^1S\to T^1S$ be the geodesic flow on the unitary tangent bundle of $S$. 

Consider $\alpha$ and $\beta$ two simple, closed, oriented and disjoint geodesics in $S$. Let us see them as periodic orbits $\alpha,\beta:[0,1]\to T^1S$ of the flow $\varphi_t$. 

It is a standard procedure to make a DA-type perturbation of the vector field $\frac{\partial \varphi_t}{\partial t}|_{t=0}$ in a neighborhood of $\alpha$ in order to transform $\alpha$ into a repelling periodic orbit for a new flow $\psi_t:T^1S\to T^1S$ such that $\frac{\partial \varphi_t}{\partial t}|_{t=0}$ and $\frac{\partial \psi_t}{\partial t}|_{t=0}$ coincide outside a small neighborhood of $\alpha$.

By considering then $T$ a small toroidal neighborhood of $\alpha$ such that $\psi_t$ points inward into $T^1S\setminus T$ along the boundary $\partial T$ we obtain that the maximal invariant set of $\psi_t|_{T^1S\setminus T}$ is a connected attracting hyperbolic set $\Lambda \subset T^1S\setminus T$. By cutting out $T$ from $T^1S$ and gluing back adequatly another copy of $\psi_t|_{T^1S\setminus T}$ with the inverse orientation one can obtain a non-transitive Anosov flow with $\Lambda$ as its unique attracting basic piece (see the techniques in \cite{FW} and \cite{BBY} for all the details).

We claim now that $\psi_t|_\Lambda$ is not orbit equivalent to a suspension. Suppose by contradiction that it is. Then we can consider $\rho:\Lambda\to S^1$ such that  $\lim_{t\to +\infty}\widetilde{\rho\circ \varphi}_t(x)=+\infty$ and $\lim_{t\to -\infty}\widetilde{\rho\circ \varphi}_t(x)=-\infty$ for every $x\in \Lambda$, where $t\mapsto \widetilde{\rho\circ \varphi}_t(x):\mathbb{R}\to \mathbb{R}$ is any lift of $t\mapsto \rho\circ \varphi_t(x):\mathbb{R}\to S^1$.

We can extend $\rho$ to a small open $\varphi_t$-forward invariant neighborhood $U$ of $\Lambda$ such that $\lim_{t\to +\infty}\widetilde{\rho\circ \varphi}_t(x)=+\infty$ continues to be valid for every $x\in U$ (see the proof of Theorem \ref{thmA'} for details on how to construct such an $U$). By considering an adapted metric such that $\varphi_t$ contracts distances inside strong stable leaves for all future iterates we can take $U$ of the form $\bigcup_{x\in \Lambda}\W^s_\delta(x)$ for some $\delta>0$. In particular, $\varphi_t$ points inwards to $U$ in every point of $\partial U$. 

We can extend $\rho$ continuously to $T^1S\setminus T$ by setting $\rho(y)=\rho(\varphi_{t_y}(y))$ for every $y\in T^1S\setminus (T\cup U)$ where $t_y$ is the unique non-negative time such that $\varphi_{t_y}(y)\in \partial U$.

Now, $\beta:[0,1]\to \Lambda$ is freely homotopic to its inverse $\beta^{-1}:[0,1]\to \Lambda$ in $T^1S$ by the homotopy $\beta_s$ with $s\in[0,1]$ that for each $t$ takes $\dot{\beta}(t)$ and rotates it clockwise $s\pi$. As $\beta_s$ coincides with $\beta$ in the base $S$, we can consider $T$ sufficiently close to $\alpha$ so that this homotopy takes place inside $T^1S\setminus T$. This homotopy gives an homotopy between the curve $t\mapsto \widetilde{\rho\circ\beta}(t):\mathbb{R}\to \mathbb{R}$ that lifts $t\mapsto \rho \circ \beta(t):\mathbb{R}\to S^1$ and the curve $t\mapsto \widetilde{\rho\circ\beta^{-1}}(t):\mathbb{R}\to \mathbb{R}$ that lifts $t\mapsto \rho \circ \beta^{-1}(t):\mathbb{R}\to S^1$. This is an imposible homotopy since $\lim_{t\to +\infty}\widetilde{\rho\circ\beta}(t)=+\infty$ and $\lim_{t\to +\infty}\widetilde{\rho\circ\beta^{-1}}(t)=-\infty$. We get to a contradiction and the claim is proved.
\end{example}

\end{document}